\documentclass{amsart}
\usepackage{amsfonts}

\newtheorem{theorem}{Theorem}[section]
\newtheorem{lemma}[theorem]{Lemma}
\theoremstyle{definition}

\theoremstyle{remark}

\numberwithin{equation}{section}
\theoremstyle{plain}

\newtheorem{corollary}{Corollary}

\newtheorem{proposition}{Proposition}

\copyrightinfo{2001}{enter name of copyright holder}
\input{tcilatex}

\begin{document}
\title[]{Rigidity of Gradient Ricci Solitons}
\author{Peter Petersen}
\address{520 Portola Plaza\\
Dept of Math UCLA\\
Los Angeles, CA 90095}
\email{petersen@math.ucla.edu}
\urladdr{http://www.math.ucla.edu/\symbol{126}petersen}
\thanks{}
\author{William Wylie}
\email{wylie@math.ucla.edu}
\urladdr{http://www.math.ucla.edu/\symbol{126}wylie}
\date{}

\begin{abstract}
We define a gradient Ricci soliton to be rigid if it is a flat bundle $%
N\times _{\Gamma }\mathbb{R}^{k}$ where $N$ is Einstein. It is known that
not all gradient solitons are rigid. Here we offer several natural
conditions on the curvature that characterize rigid gradient solitons. Other
related results on rigidity of Ricci solitons are also explained in the last
section.
\end{abstract}

\maketitle

\section{Introduction}

A Ricci soliton is a Riemannian metric together with a vector field $\left(
M,g,X\right) $ that satisfies%
\begin{equation*}
\mathrm{Ric}+\frac{1}{2}L_{X}g=\lambda g.
\end{equation*}%
It is called shrinking when $\lambda >0,$ steady when $\lambda =0$, and
expanding when $\lambda <0$. In case $X=\nabla f$ the equation can also be
written as%
\begin{equation*}
\mathrm{Ric}+\mathrm{Hess}f=\lambda g
\end{equation*}%
and is called a gradient (Ricci) soliton. We refer the reader to \cite%
{Cao2006, Chow-Knopf, Chow-Lu-Ni, Derdzinski} for background on Ricci
solitons and their connection to the Ricci flow. It is also worth pointing
out that Perel'man has shown that on a compact manifold Ricci solitions are
always gradient solitons, see \cite{PerelmanI}.

Clearly Einstein metrics are solitons with $f$ being trivial. Another
interesting special case occurs when $f=\frac{\lambda }{2}\left\vert
x\right\vert ^{2}$ on $\mathbb{R}^{n}.$ In this case%
\begin{equation*}
\mathrm{Hess}f=\lambda g
\end{equation*}%
and therefore yields a gradient soliton where the background metric is flat.
This example is called a \emph{Gaussian}. Taking a product $N\times \mathbb{R%
}^{k}$ with $N$ being Einstein with Einstein constant $\lambda $ and $f=%
\frac{\lambda }{2}\left\vert x\right\vert ^{2}$ on $\mathbb{R}^{k}$ yields a
mixed gradient soliton. We can further take a quotient $N\times _{\Gamma }%
\mathbb{R}^{k}$, where $\Gamma $ acts freely on $N$ and by orthogonal
transformations on $\mathbb{R}^{k}$ (no translational components) to get a
flat vector bundle over a base that is Einstein and with $f=\frac{\lambda }{2%
}d^{2}$ where $d$ is the distance in the flat fibers to the base.

We say that a gradient soliton is \emph{rigid} if it is of the type $N\times
_{\Gamma }\mathbb{R}^{k}$ just described.

The goal of this paper is to determine when gradient solitons are rigid. For
compact manifolds it is easy to see that they are rigid precisely when the
scalar curvature is constant see \cite{Eminenti-LaNave-Mantegazza}. In fact
we can show something a bit more general

\begin{theorem}
A compact gradient soliton is rigid with trivial $f$ if%
\begin{equation*}
\mathrm{Ric}\left( \nabla f,\nabla f\right) \leq 0.
\end{equation*}
\end{theorem}

Moreover, in dimensions 2 \cite{Hamilton1988} and 3 \cite{Ivey1993} all
compact solitons are rigid. There are compact shrinking (K\"{a}hler)
gradient solitons in dimension 4 that do not have constant scalar curvature,
the first example was constructed by Koiso \cite{Koiso1990} see also \cite%
{Cao1996, Wang-Zhu2004}. It is also not hard to see that, in any dimension,
compact steady or expanding solitons are rigid (see \cite{Ivey1993} and
Corollary \ref{Cmpct}). In fact, at least in the steady gradient soliton
case, this seems to go back to Lichnerowicz, see section 3.10 of \cite%
{Bourguignon1981}.

In the noncompact case Perel'man has shown that all 3-dimensional shrinking
gradient solitons with nonnegative sectional curvature are rigid \cite%
{PerelmanII}. However, in higher dimensions, it is less clear how to detect
rigidity. In fact there are expanding Ricci solitons with constant scalar
curvature that are not rigid in the above sense. These spaces are left
invariant metrics on nilpotent groups constructed by Lauret \cite{Lauret2001}
that are not gradient solitions. For other examples of noncompact gradient
solitons with large symmetry groups see \cite{Cao1996, Cao1997,
Feldman-Ilmanen-Knopf2003, Ivey1994}.

Note that if a soliton is rigid, then the \textquotedblleft
radial\textquotedblright\ curvatures vanish, i.e., 
\begin{equation*}
R\left( \cdot ,\nabla f\right) \nabla f=0,
\end{equation*}%
and the scalar curvature is constant. Conversely we just saw that constant
scalar curvature and radial Ricci flatness: $\mathrm{Ric}\left( \nabla
f,\nabla f\right) =0$ each imply rigidity on compact solitons. In the
noncompact case we can show

\begin{theorem}
\label{main} A shrinking (expanding) gradient soliton 
\begin{equation*}
\mathrm{Ric}+\mathrm{Hess}f=\lambda g
\end{equation*}%
is rigid if and only if it has constant scalar curvature and is radially
flat, i.e., $\mathrm{sec}\left( E,\nabla f\right) =0$.
\end{theorem}

While radial flatness seems like a strong assumption, there are a number of
weaker conditions that guarantee radial flatness.

\begin{proposition}
The following conditions for a shrinking (expanding) gradient soliton 
\begin{equation*}
\mathrm{Ric}+\mathrm{Hess}f=\lambda g
\end{equation*}%
all imply that the metric is radially flat and has constant scalar curvature

\begin{enumerate}
\item The scalar curvature is constant and $\mathrm{sec}\left( E,\nabla
f\right) \geq 0$ $\left( \mathrm{sec}\left( E,\nabla f\right) \leq 0.\right) 
$

\item The scalar curvature is constant and $0\leq \mathrm{Ric}\leq \lambda g$
$\left( \lambda g\leq \mathrm{Ric}\leq 0.\right) $

\item The curvature tensor is harmonic.

\item $\mathrm{Ric\ }\geq 0$ $\left( \mathrm{Ric}\leq 0\right) $ and $%
\mathrm{sec}\left( E,\nabla f\right) =0.$
\end{enumerate}
\end{proposition}

Given the above theorem it is easy to see that rigid solitions also satisfy
these conditions.

Condition 2 is very similar to a statement by Naber, but our proof is quite
different. The following result shows that, for shrinking solitons, the
scalar curvature condition is in fact redundant. Thus we are offering an
alternate proof for part of Naber's result (see \cite{Naber}).

\begin{lemma}[Naber]
\label{NaberLemma} If $M$ is a shrinking gradient Ricci Soliton with $0\leq 
\mathrm{Ric}\leq \lambda g,$ then the scalar curvature is constant.
\end{lemma}

There is an interesting relationship between this result and Perel'man's
classification in dimension 3. The main part of the classification is to
show that there are no noncompact shrinking gradient solitons with positive
sectional curvature. Perel'man's proof has two parts, first he shows that
such a metric has $\mathrm{scal} \leq 2 \lambda$ and then he uses this fact,
and the Gauss-Bonnet theorem, to arrive at a contradiction. It is a simple
algebraic fact that if $\mathrm{sec} \geq 0$ and $\mathrm{scal} \leq 2
\lambda$ then $\mathrm{Ric} \leq \lambda$. Therefore, Naber's lemma implies
the following gap theorem which generalizes the second part of Perel'man's
argument to higher dimensions.

\begin{theorem}
If $M^n$ is a shrinking gradient Ricci soliton with nonnegative sectional
curvature and $\mathrm{scal} \leq 2 \lambda$ then the universal cover of $M$
is isometric to either $\mathbb{R}^n$ or $S^2 \times \mathbb{R}^{n-2}$.
\end{theorem}

The key to most of our proofs rely on a new equation that in a fairly
obvious way relates rigidity, radial curvatures, and scalar curvature%
\begin{equation*}
\nabla _{\nabla f}\mathrm{Ric}+\mathrm{Ric}\circ \left( \lambda I-\mathrm{Ric%
}\right) =R\left( \cdot ,\nabla f\right) \nabla f+\frac{1}{2}\nabla _{\cdot
}\nabla \mathrm{scal.}
\end{equation*}

While we excluded steady solitons from the above result, it wasn't really
necessary to do so. In fact it is quite easy to prove something that sounds
more general.

\begin{theorem}
A steady soliton 
\begin{equation*}
\mathrm{Ric}+\mathrm{Hess}f=0
\end{equation*}%
whose scalar curvature achieves its minimum is Ricci flat. In particular,
steady gradient solitions with constant scalar curvature are Ricci flat.
\end{theorem}

In the context of condition 3 about harmonicity of the curvature there is a
rather interesting connection with gradient solitons. Consider the exterior
covariant derivative%
\begin{equation*}
d^{\nabla }:\Omega ^{p}\left( M,TM\right) \rightarrow \Omega ^{p+1}\left(
M,TM\right)
\end{equation*}%
for forms with values in the tangent bundle. The curvature can then be
interpreted as the 2-form%
\begin{equation*}
R\left( X,Y\right) Z=\left( \left( d^{\nabla }\circ d^{\nabla }\right)
\left( Z\right) \right) \left( X,Y\right)
\end{equation*}%
and Bianchi's second identity as $d^{\nabla }R=0$. The curvature is harmomic
if $d^{\ast }R=0,$ where $d^{\ast }$ is the adjoint of $d^{\nabla }.$ If we
think of $\mathrm{Ric}$ as a 1-form with values in $TM,$ then Bianchi's
second identity implies%
\begin{equation*}
d^{\nabla }\mathrm{Ric}=-d^{\ast }R.
\end{equation*}%
Thus the curvature tensor is harmonic if and only if the Ricci tensor is
closed. This condition has been studied extensively as a generalization of
being an Einstein metric (see \cite{Besse}, Chapter 16). It is also easy to
see that it implies constant scalar curvature.

Next note that the condition for being a steady gradient soliton is the same
as saying that the Ricci tensor is exact%
\begin{equation*}
\mathrm{Ric}=d^{\nabla }\left( -X\right) =-\nabla X.
\end{equation*}%
Since the Ricci tensor is symmetric, this requires that $X$ is locally a
gradient field. The general gradient soliton equation%
\begin{equation*}
\mathrm{Ric}=d^{\nabla }\left( -X\right) +\lambda I
\end{equation*}%
then appears to be a simultaneous generalization of being Einstein and
exact. Thus Theorem \ref{main} implies that rigid gradient solitons are
precisely those metrics that satisfy all the generalized Einstein conditions.

Throughout the paper we also establish several other simple results that
guarantee rigidity under slightly different assumptions on the curvature and
geometry of the space. We can also use the techniques developed here to
obtain some results for solitons with large amounts symmetry, this will be
the topic of a forthcoming paper.

\section{Formulas}

In this section we establish the general formulas that will used to prove
the various rigidity results we are after. There are two sets of results.
The most general and weakest for Ricci solitons and the more interesting and
powerful for gradient solitons.

First we establish a general formula that leads to the Bochner formulas for
Killing and gradient fields (see also \cite{Poor}.)

\begin{lemma}
On a Riemannian manifold%
\begin{equation*}
\mathrm{div}\left( L_{X}g\right) \left( X\right) =\frac{1}{2}\Delta
\left\vert X\right\vert ^{2}-\left\vert \nabla X\right\vert ^{2}+\mathrm{Ric}%
\left( X,X\right) +D_{X}\mathrm{div}X
\end{equation*}%
When $X=\nabla f$ is a gradient field we have%
\begin{equation*}
\left( \mathrm{div}L_{X}g\right) \left( Z\right) =2\mathrm{Ric}\left(
Z,X\right) +2D_{Z}\mathrm{div}X
\end{equation*}%
or in $\left( 1,1\right) $-tensor notation%
\begin{equation*}
\mathrm{div}\nabla \nabla f=\mathrm{Ric}\left( \nabla f\right) +\nabla
\Delta f
\end{equation*}
\end{lemma}

\begin{proof}
We calculate with a frame that is parallel at $p$%
\begin{eqnarray*}
&&\mathrm{div}\left( L_{X}g\right) \left( X\right) \\
&=&\left( \nabla _{E_{i}}L_{X}g\right) \left( E_{i},X\right) \\
&=&\nabla _{E_{i}}\left( L_{X}g\left( E_{i},X\right) \right) -L_{X}g\left(
E_{i},\nabla _{E_{i}}X\right) \\
&=&\nabla _{E_{i}}\left( g\left( \nabla _{E_{i}}X,X\right) +g\left(
E_{i},\nabla _{X}X\right) \right) -g\left( \nabla _{E_{i}}X,\nabla
_{E_{i}}X\right) -g\left( E_{i},\nabla _{\nabla _{E_{i}}X}X\right) \\
&=&\Delta \frac{1}{2}\left\vert X\right\vert ^{2}+\nabla _{E_{i}}g\left(
E_{i},\nabla _{X}X\right) -\left\vert \nabla X\right\vert ^{2}-g\left(
E_{i},\nabla _{\nabla _{E_{i}}X}X\right) \\
&=&\Delta \frac{1}{2}\left\vert X\right\vert ^{2}-\left\vert \nabla
X\right\vert ^{2}+g\left( \nabla _{E_{i},X}^{2}X,E_{i}\right) \\
&=&\Delta \frac{1}{2}\left\vert X\right\vert ^{2}-\left\vert \nabla
X\right\vert ^{2}+\mathrm{Ric}\left( X,X\right) +g\left( \nabla
_{X,E_{i}}^{2}X,E_{i}\right) \\
&=&\Delta \frac{1}{2}\left\vert X\right\vert ^{2}-\left\vert \nabla
X\right\vert ^{2}+\mathrm{Ric}\left( X,X\right) +D_{X}\mathrm{div}X
\end{eqnarray*}%
And when $Z\rightarrow \nabla _{Z}X$ is self-adjoint we have 
\begin{eqnarray*}
&&\left( \mathrm{div}L_{X}g\right) \left( Z\right) \\
&=&\left( \nabla _{E_{i}}L_{X}g\right) \left( E_{i},Z\right) \\
&=&\nabla _{E_{i}}\left( L_{X}g\left( E_{i},Z\right) \right) -L_{X}g\left(
E_{i},\nabla _{E_{i}}Z\right) \\
&=&\nabla _{E_{i}}\left( g\left( \nabla _{E_{i}}X,Z\right) +g\left(
E_{i},\nabla _{Z}X\right) \right) -g\left( \nabla _{E_{i}}X,\nabla
_{E_{i}}Z\right) -g\left( E_{i},\nabla _{\nabla _{E_{i}}Z}X\right) \\
&=&\nabla _{E_{i}}\left( g\left( \nabla _{E_{i}}X,Z\right) +g\left(
E_{i},\nabla _{Z}X\right) \right) -g\left( \nabla _{E_{i}}X,\nabla
_{E_{i}}Z\right) -g\left( E_{i},\nabla _{\nabla _{E_{i}}Z}X\right) \\
&=&\nabla _{E_{i}}\left( g\left( \nabla _{Z}X,E_{i}\right) +g\left(
E_{i},\nabla _{E_{i}}\nabla _{Z}X\right) \right) -g\left( \nabla
_{E_{i}}X,\nabla _{E_{i}}Z\right) -g\left( E_{i},\nabla _{\nabla
_{E_{i}}Z}X\right) \\
&=&2g\left( \nabla _{E_{i},Z}^{2}X,E_{i}\right) \\
&=&2\mathrm{Ric}\left( Z,X\right) +2g\left( \nabla
_{Z,E_{i}}^{2}X,E_{i}\right) \\
&=&2\mathrm{Ric}\left( Z,X\right) +2D_{Z}\mathrm{div}X
\end{eqnarray*}
\end{proof}

\begin{corollary}
If $X$ is a Killing field, then%
\begin{equation*}
\Delta \frac{1}{2}\left\vert X\right\vert ^{2}=\left\vert \nabla
X\right\vert ^{2}-\mathrm{Ric}\left( X,X\right)
\end{equation*}
\end{corollary}

\begin{proof}
Use that $L_{X}g=0=\mathrm{div}X$ in the above formula.
\end{proof}

\begin{corollary}
If $X$ is a gradient field, then%
\begin{equation*}
\Delta \frac{1}{2}\left\vert X \right\vert ^{2}=\left\vert \nabla
X\right\vert ^{2}+D_{X}\mathrm{div}X+\mathrm{Ric}\left( X,X\right)
\end{equation*}
\end{corollary}

\begin{proof}
Let $Z=X$ in the second equation above and equate them to get the formula.
\end{proof}

We are now ready to derive formulas for Ricci solitons%
\begin{equation*}
\mathrm{Ric}+\frac{1}{2}L_{X}g=\lambda g
\end{equation*}

\begin{lemma}
\label{Sol1} A Ricci soliton satisfies%
\begin{equation*}
\frac{1}{2}\left( \Delta -D_{X}\right) \left\vert X\right\vert
^{2}=\left\vert \nabla X\right\vert ^{2}-\lambda \left\vert X\right\vert ^{2}
\end{equation*}
\end{lemma}

\begin{proof}
The trace of the soliton equation says that%
\begin{equation*}
\mathrm{scal}+\mathrm{div}X=n\lambda
\end{equation*}%
so%
\begin{equation*}
D_{Z}\mathrm{scal}=-D_{Z}\mathrm{div}X
\end{equation*}%
The contracted second Bianchi identity that forms the basis for Einstein's
equations says that%
\begin{equation*}
D_{Z}\mathrm{scal}=2\mathrm{div}\mathrm{Ric}\left( Z\right)
\end{equation*}%
Using $Z=X$ and the soliton equation then gives%
\begin{eqnarray*}
-D_{X}\mathrm{div}X &=&2\mathrm{div}\mathrm{Ric}\left( X\right) \\
&=&-\mathrm{div}\left( L_{X}g\right) \left( X\right) \\
&=&-\left( \frac{1}{2}\Delta \left\vert X\right\vert ^{2}-\left\vert \nabla
X\right\vert ^{2}+\mathrm{Ric}\left( X,X\right) +D_{X}\mathrm{div}X\right)
\end{eqnarray*}%
Thus%
\begin{eqnarray*}
\frac{1}{2}\Delta \left\vert X\right\vert ^{2} &=&\left\vert \nabla
X\right\vert ^{2}-\mathrm{Ric}\left( X,X\right) \\
&=&\left\vert \nabla X\right\vert ^{2}+\frac{1}{2}\left( L_{X}g\right)
\left( X,X\right) -\lambda \left\vert X\right\vert ^{2} \\
&=&\left\vert \nabla X\right\vert ^{2}+\frac{1}{2}D_{X}\left\vert
X\right\vert ^{2}-\lambda \left\vert X\right\vert ^{2}
\end{eqnarray*}%
from which we get the equation.
\end{proof}

We now turn our attention to gradient solitons. In this case we can use $%
\left( 1,1\right) $-tensors and write the soliton equation as%
\begin{equation*}
\mathrm{Ric}+\nabla \nabla f=\lambda I
\end{equation*}%
or in condensed form%
\begin{eqnarray*}
\mathrm{Ric}+S &=&\lambda I, \\
S &=&\nabla \nabla f
\end{eqnarray*}%
With this notation we can now state and prove some interesting formulas for
the scalar curvature of gradient solitons. The first and last are known (see 
\cite{Chow-Lu-Ni}), while the middle ones seem to be new.

\begin{lemma}
\label{Sol2} A gradient soliton satisfies%
\begin{equation*}
\nabla \mathrm{scal}=2\mathrm{Ric}\left( \nabla f\right)
\end{equation*}%
\begin{eqnarray*}
\nabla _{\nabla f}S+S\circ \left( S-\lambda I\right) &=&-R\left( \cdot
,\nabla f\right) \nabla f-\frac{1}{2}\nabla _{\cdot }\nabla \mathrm{scal}, \\
\nabla _{\nabla f}\mathrm{Ric}+\mathrm{Ric}\circ \left( \lambda I-\mathrm{Ric%
}\right) &=&R\left( \cdot ,\nabla f\right) \nabla f+\frac{1}{2}\nabla
_{\cdot }\nabla \mathrm{scal}
\end{eqnarray*}%
\begin{equation*}
\frac{1}{2}\left( \Delta -D_{\nabla f}\right) \mathrm{scal}=\frac{1}{2}%
\Delta _{f}\mathrm{scal}=\mathrm{tr}\left( \mathrm{Ric}\circ \left( \lambda
I-\mathrm{Ric}\right) \right)
\end{equation*}
\end{lemma}

\begin{proof}
We have the Bochner formula%
\begin{equation*}
\mathrm{div}\left( \nabla \nabla f\right) =\mathrm{Ric}\left( \nabla
f\right) +\nabla \Delta f
\end{equation*}%
The trace of the soliton equation gives%
\begin{eqnarray*}
\mathrm{scal}+\Delta f &=&n\lambda , \\
\nabla \mathrm{scal}+\nabla \Delta f &=&0
\end{eqnarray*}%
while the divergence of the soliton equation gave us%
\begin{equation*}
\mathrm{div}\mathrm{Ric}+\mathrm{div}\left( \nabla \nabla f\right) =0
\end{equation*}%
Together this yields%
\begin{eqnarray*}
\nabla \mathrm{scal} &=&2\mathrm{div}\mathrm{Ric} \\
&=&-2\mathrm{div}\left( \nabla \nabla f\right) \\
&=&-2\mathrm{Ric}\left( \nabla f\right) -2\nabla \Delta f \\
&=&-2\mathrm{Ric}\left( \nabla f\right) +2\nabla \mathrm{scal}
\end{eqnarray*}%
and hence the first formula.

Using this one can immediately find a formula for the Laplacian of the
scalar curvature. However our goal is the establish the second set of
formulas. The last formula is then obtained by taking traces.

We use the equation%
\begin{equation*}
R\left( E,\nabla f\right) \nabla f=\nabla _{E,\nabla f}^{2}\nabla f-\nabla
_{\nabla f,E}^{2}\nabla f
\end{equation*}%
The second term on the right%
\begin{equation*}
\nabla _{\nabla f,E}^{2}\nabla f=\left( \nabla _{\nabla f}S\right) \left(
E\right)
\end{equation*}%
while the first can be calculated%
\begin{eqnarray*}
\nabla _{E,\nabla f}^{2}\nabla f &=&-\left( \nabla _{E}\mathrm{Ric}\right)
\left( \nabla f\right) \\
&=&-\nabla _{E}\mathrm{Ric}\left( \nabla f\right) +\mathrm{Ric}\left( \nabla
_{E}\nabla f\right) \\
&=&-\frac{1}{2}\nabla _{E}\nabla \mathrm{scal}+\mathrm{Ric}\circ S\left(
E\right) \\
&=&-\frac{1}{2}\nabla _{E}\nabla \mathrm{scal}+\left( \lambda I-S\right)
\circ S\left( E\right) \\
&=&-\frac{1}{2}\nabla _{E}\nabla \mathrm{scal}+\mathrm{Ric}\circ \left(
\lambda I-\mathrm{Ric}\right)
\end{eqnarray*}%
This yields the set of formulas in the middle.

Taking traces in%
\begin{equation*}
\nabla _{\nabla f}\mathrm{Ric}+\mathrm{Ric}\circ \left( \lambda I-\mathrm{Ric%
}\right) =R\left(E ,\nabla f\right) \nabla f+\frac{1}{2}\nabla _{E}\nabla 
\mathrm{scal}
\end{equation*}%
yields%
\begin{equation*}
\nabla _{\nabla f}\mathrm{scal}+\mathrm{tr}\left( \mathrm{Ric}\circ \left(
\lambda I-\mathrm{Ric}\right) \right) =\mathrm{Ric}\left( \nabla f,\nabla
f\right) +\frac{1}{2}\Delta \mathrm{scal}
\end{equation*}%
Since%
\begin{equation*}
\mathrm{Ric}\left( \nabla f,\nabla f\right) =\frac{1}{2}D_{\nabla f}\mathrm{%
scal}
\end{equation*}%
we immediately get the last equation.
\end{proof}

Note that if $\lambda _{i}$ are the eigenvalues of the Ricci tensor then the
last equation can be rewritten in several useful ways%
\begin{eqnarray*}
\frac{1}{2}\Delta _{f}\mathrm{scal} &=&\mathrm{tr}\left( \mathrm{Ric}\circ
\left( \lambda I-\mathrm{Ric}\right) \right) \\
&=&\sum \lambda _{i}\left( \lambda -\lambda _{i}\right) \\
&=&-\left\vert \mathrm{Ric}\right\vert ^{2}+\lambda \mathrm{scal} \\
&=&-\left\vert \mathrm{Ric}-\frac{1}{n}\mathrm{scal}g\right\vert ^{2}+%
\mathrm{scal}\left( \lambda -\frac{1}{n}\mathrm{scal}\right)
\end{eqnarray*}

\section{Rigidity Characterization}

We start with a motivational appetizer on rigidity of gradient solitons.

\begin{proposition}
\label{Ein-Sol} A gradient soliton which is Einstein, either has $\mathrm{%
Hess}f=0$ or is a Gaussian.
\end{proposition}

\begin{proof}
Assume that 
\begin{equation*}
\mu g+\mathrm{Hess}f=\lambda g.
\end{equation*}%
If $\mu =\lambda ,$ then the Hessian vanishes. Otherwise we have that the
Hessian is proportional to $g.$ Multiplying $f$ by a constant then leads us
to a situation where 
\begin{equation*}
\mathrm{Hess}f=g.
\end{equation*}%
This shows that $f$ is a proper strictly convex function. By adding a
suitable constant to $f$ we also see that $r=\sqrt{f}$ is a distance
function from the unique minimum of $f.$ It is now easy to see that the
radial curvatures vanish and then that the space is flat (see also \cite%
{Petersen})
\end{proof}

Next we dispense with rigidity for compact solitons.

\begin{theorem}
A compact Ricci soliton with 
\begin{equation*}
\mathrm{Ric}\left( X,X\right) \leq 0
\end{equation*}%
is Einstein with Einstein constant $\lambda .$ In particular, compact
gradient solitions with constant scalar curvature are Einstein.
\end{theorem}

\begin{proof}
We have a Ricci soliton%
\begin{equation*}
\mathrm{Ric}+L_{X}g=\lambda g.
\end{equation*}%
The Laplacian of $X$ then satisfies%
\begin{eqnarray*}
\Delta \frac{1}{2}\left\vert X\right\vert ^{2} &=&\left\vert \nabla
X\right\vert ^{2}-\mathrm{Ric}\left( X,X\right) \\
&\geq &0
\end{eqnarray*}%
The divergence theorem then shows that $\nabla X$ vanishes. In particular $%
L_{X}g=0.$

The second part is a simple consequence of having $X=\nabla f$ and the
equation%
\begin{equation*}
D_{\nabla f}\mathrm{scal}=2\mathrm{Ric}\left( \nabla f,\nabla f\right) .
\end{equation*}
\end{proof}

We also note that, when the Ricci tensor has a definite sign, having zero
radial Ricci curvature is equivalent to having constant scalar curvature. In
particular this implies the equivalence of condition (4) in Proposition 1.

\begin{proposition}
A gradient soliton with nonnegative (or nonpositive) Ricci curvature has
constant scalar curvature if and only if $\mathrm{Ric}(\nabla f, \nabla f)=0$%
.
\end{proposition}

\begin{proof}
We know from elementary linear algebra that, for a nonnegative (or
nonpositive) definite, self-adjoint operator $T$, 
\begin{equation*}
\langle Tv,v \rangle = 0 \qquad \Rightarrow \qquad Tv = 0.
\end{equation*}
So the proposition follows easily by taking $T$ to be the $(1,1)$-Ricci
tensor and the fact that $\nabla \mathrm{scal} = 2 \mathrm{Ric}(\nabla f)$
for a gradient soliton.
\end{proof}

Steady solitons are also easy to deal with

\begin{proposition}
A steady gradient soliton with constant scalar curvature is Ricci flat.
Moreover, if $f$ is not constant then it is a product of a Ricci flat
manifold with $\mathbb{R}.$
\end{proposition}

\begin{proof}
First we note that 
\begin{eqnarray*}
0 &=&\frac{1}{2}\Delta _{f}\mathrm{scal} \\
&=&-\left\vert \mathrm{Ric}-\frac{1}{n}\mathrm{scal}g\right\vert ^{2}+%
\mathrm{scal}\left( \lambda -\frac{1}{n}\mathrm{scal}\right) \\
&=&-\left\vert \mathrm{Ric}-\frac{1}{n}\mathrm{scal}g\right\vert ^{2}-\frac{1%
}{n}\mathrm{scal}^{2} \\
&\leq &0
\end{eqnarray*}%
Thus $\mathrm{scal}=0$ and $\mathrm{Ric}=0.$ This shows that $\mathrm{Hess}%
f=0.$ Thus $f$ is either constant or the manifold splits along the gradient
of $f.$
\end{proof}

This partly motivates our next result.

\begin{proposition}
\label{Prop} Assume that we have a gradient soliton%
\begin{equation*}
\mathrm{Ric}+\mathrm{Hess}f=\lambda g
\end{equation*}%
with constant scalar curvature and $\lambda \neq 0.$ When $\lambda >0$ we
have $0\leq \mathrm{scal}\leq n\lambda .$ When $\lambda <0$ we have $%
n\lambda \leq \mathrm{scal}\leq 0$. In either case the metric is Einstein
when the scalar curvature equals either of the extreme values.
\end{proposition}

\begin{proof}
Again we have that%
\begin{equation*}
0=\frac{1}{2}\Delta _{f}\mathrm{scal}=-\left\vert \mathrm{Ric}-\frac{1}{n}%
\mathrm{scal}g\right\vert ^{2}+\mathrm{scal}\left( \lambda -\frac{1}{n}%
\mathrm{scal}\right)
\end{equation*}%
showing that%
\begin{equation*}
0\leq \left\vert \mathrm{Ric}-\frac{1}{n}\mathrm{scal}g\right\vert ^{2}=%
\mathrm{scal}\left( \lambda -\frac{1}{n}\mathrm{scal}\right)
\end{equation*}

Thus $\mathrm{scal}\in \left[ 0,n\lambda \right] $ if the soliton is
shrinking and the metric is Einstein if the scalar curvature takes on either
of the boundary values. A similar analysis holds in the expanding case.
\end{proof}

Before proving the main characterization we study the conditions that
gurantee radial flatness.

\begin{proposition}
The following conditions for a shrinking (expanding) gradient soliton 
\begin{equation*}
\mathrm{Ric}+\mathrm{Hess}f=\lambda g
\end{equation*}%
all imply that it is radially flat.

\begin{enumerate}
\item The scalar curvature is constant and $\mathrm{sec}\left( E,\nabla
f\right) \geq 0$ $\left( \mathrm{sec}\left( E,\nabla f\right) \leq 0.\right) 
$

\item The scalar curvature is constant and $0\leq \mathrm{Ric}\leq \lambda g$
$\left( \lambda g\leq \mathrm{Ric}\leq 0.\right) $

\item The curvature tensor is harmonic.
\end{enumerate}
\end{proposition}

\begin{proof}
1: Use the equations%
\begin{eqnarray*}
0 &=&\frac{1}{2}\nabla _{\nabla f}\mathrm{scal}=\mathrm{Ric}\left( \nabla
f,\nabla f\right) \\
&=&\sum g\left( R\left( E_{i},\nabla f\right) \nabla f,E_{i}\right)
\end{eqnarray*}%
to see that $g\left( R\left( E_{i},\nabla f\right) \nabla f,E_{i}\right) =0$
if the radial curvatures are always nonnegative (nonpositive).

2: First observe that 
\begin{equation*}
0=\frac{1}{2}\Delta _{f}\mathrm{scal}=\mathrm{tr}\left( \mathrm{Ric}\circ
\left( \lambda I-\mathrm{Ric}\right) \right)
\end{equation*}%
The assumptions on the Ricci curvature imply that $\mathrm{Ric}\circ \left(
\lambda I-\mathrm{Ric}\right) $ is a nonnegative operator. Thus%
\begin{equation*}
\mathrm{Ric}\circ \left( \lambda I-\mathrm{Ric}\right) =0.
\end{equation*}%
This shows that the only possible eigenvalues for $\mathrm{Ric}$ and $\nabla
\nabla f$ are $0$ and $\lambda .$

To establish radial flatness we then use that the formula 
\begin{equation*}
\nabla _{\nabla f}\mathrm{Ric}+\mathrm{Ric}\circ \left( \lambda I-\mathrm{Ric%
}\right) =R\left( \cdot ,\nabla f\right) \nabla f+\frac{1}{2}\nabla _{\cdot
}\nabla \mathrm{scal}
\end{equation*}%
is reduced to%
\begin{eqnarray*}
R\left( \cdot ,\nabla f\right) \nabla f &=&\nabla _{\nabla f}\mathrm{Ric} \\
&=&-\nabla _{\nabla f,\cdot }^{2}\nabla f
\end{eqnarray*}%
Next pick a field $E$ such that $\nabla _{E}\nabla f=0,$ then%
\begin{eqnarray*}
g\left( \nabla _{\nabla f,E}^{2}\nabla f,E\right) &=&g\left( \nabla _{\nabla
f}\nabla _{E}\nabla f,E\right) -g\left( \nabla _{\nabla _{\nabla f}E}\nabla
f,E\right) \\
&=&-g\left( \nabla _{E}\nabla f,\nabla _{\nabla f}E\right) \\
&=&0
\end{eqnarray*}%
and finally when $\nabla _{E}\nabla f=\lambda E$%
\begin{eqnarray*}
g\left( \nabla _{\nabla f,E}^{2}\nabla f,E\right) &=&g\left( \nabla _{\nabla
f}\nabla _{E}\nabla f,E\right) -g\left( \nabla _{\nabla _{\nabla f}E}\nabla
f,E\right) \\
&=&\lambda g\left( \nabla _{\nabla f}E,E\right) -g\left( \nabla _{E}\nabla
f,\nabla _{\nabla f}E\right) \\
&=&\lambda g\left( \nabla _{\nabla f}E,E\right) -\lambda g\left( E,\nabla
_{\nabla f}E\right) \\
&=&0.
\end{eqnarray*}%
Thus $g\left( R\left( E,\nabla f\right) \nabla f,E\right) =0$ for all
eigenfields. This shows that the metric is radially flat.

3: Finally use the soliton equation to see that%
\begin{equation*}
\left( \nabla _{X}\mathrm{Ric}\right) \left( Y,Z\right) -\left( \nabla _{Y}%
\mathrm{Ric}\right) \left( X,Z\right) =-g\left( R\left( X,Y\right) \nabla
f,Z\right) .
\end{equation*}%
>From the 2nd Bianchi identity we also get that%
\begin{equation*}
\left( \nabla _{X}\mathrm{Ric}\right) \left( Y,Z\right) -\left( \nabla _{Y}%
\mathrm{Ric}\right) \left( X,Z\right) =\mathrm{div}R\left( X,Y,Z\right) =0
\end{equation*}%
since the curvature is harmonic. Thus $R\left( X,Y\right) \nabla f=0.$ In
particular $\mathrm{sec}\left( E,\nabla f\right) =0.$
\end{proof}

We now turn our attention to the main theorem. To prepare the way we show.

\begin{proposition}
Assume that we have a gradient soliton%
\begin{equation*}
\mathrm{Ric}+\mathrm{Hess}f=\lambda g
\end{equation*}%
with constant scalar curvature, $\lambda \neq 0$ and a nontrivial $f$. For a
suitable constant $\alpha $%
\begin{equation*}
f+\alpha =\frac{\lambda }{2}r^{2}
\end{equation*}%
where $r$ is a smooth function whenever $\nabla f\neq 0$ and satisfies%
\begin{equation*}
\left\vert \nabla r\right\vert =1.
\end{equation*}
\end{proposition}

\begin{proof}
Observe that%
\begin{eqnarray*}
\frac{1}{2}\nabla \left( \mathrm{scal}+\left\vert \nabla f\right\vert
^{2}\right) &=&\mathrm{Ric}\left( \nabla f\right) +\nabla _{\nabla f}\nabla f
\\
&=&\lambda \nabla f
\end{eqnarray*}%
which shows%
\begin{equation*}
\mathrm{scal}+\left\vert \nabla f\right\vert ^{2}-2\lambda f=\mathrm{const}
\end{equation*}%
By adding a suitable constant to $f$ we can then assume that%
\begin{equation*}
\left\vert \nabla f\right\vert ^{2}=2\lambda f.
\end{equation*}%
Thus $f$ has the same sign as $\lambda $ and the same zero locus as its
gradient. If we define $r$ such that%
\begin{equation*}
f=\frac{\lambda }{2}r^{2}
\end{equation*}%
then%
\begin{equation*}
\nabla f=\lambda r\nabla r
\end{equation*}%
and%
\begin{eqnarray*}
2\lambda f &=&\left\vert \nabla f\right\vert ^{2} \\
&=&\lambda ^{2}r^{2}\left\vert \nabla r\right\vert ^{2} \\
&=&2\lambda f\left\vert \nabla r\right\vert ^{2}
\end{eqnarray*}
\end{proof}

This allows us to establish our characterization of rigid gradient solitons.

\begin{theorem}
A gradient soliton 
\begin{equation*}
\mathrm{Ric}+\mathrm{Hess}f=\lambda g
\end{equation*}%
is rigid if it is radially flat and has constant scalar curvature.
\end{theorem}

\begin{proof}
We consider the case where $\lambda >0$ as the other case is similar aside
from some sign changes.

Using the condensed version of the soliton equation%
\begin{eqnarray*}
\mathrm{Ric}+S &=&\lambda I, \\
S &=&\nabla \nabla f
\end{eqnarray*}%
we have%
\begin{eqnarray*}
\nabla _{\nabla f}S+S\circ \left( S-\lambda I\right) &=&0, \\
\nabla _{\nabla f}\mathrm{Ric}+\mathrm{Ric}\circ \left( \lambda I-\mathrm{Ric%
}\right) &=&0
\end{eqnarray*}

Assume that $f=\frac{\lambda }{2}r^{2}$ where $r$ is a nonnegative distance
function. The minimum set for $f$%
\begin{equation*}
N=\left\{ x:f\left( x\right) =0\right\}
\end{equation*}%
is also characterized as%
\begin{equation*}
N=\left\{ x\in M:\nabla f\left( x\right) =0\right\}
\end{equation*}

This shows that $S\circ \left( S-\lambda I\right) =0$ on $N.$

When $r>0$ we note that the smallest eigevalue for $S$ is always absolutely
continuous and therefore satisfies the differential equation%
\begin{equation*}
D_{\nabla f}\mu _{\min }=\mu _{\min }\left( \lambda -\mu _{\min }\right) .
\end{equation*}%
We claim that $\mu _{\min }\geq 0.$ Using $r>0$ as an independent coordinate
and $\nabla f=\lambda r\nabla r$ yields%
\begin{equation*}
\partial _{r}\mu _{\min }=\frac{1}{\lambda r}\mu _{\min }\left( \lambda -\mu
_{\min }\right)
\end{equation*}%
This equation can be solved by separation of variables. In particular, $\mu
_{\min }\rightarrow -\infty $ in finite time provided $\mu _{\min }<0$
somewhere. This contradicts smoothness of $f.$ Thus we can conclude that $%
\mu _{\min }\geq 0$ and hence that $f$ is convex.

Now that we know $f$ is convex the minimum set $N$ must be totally convex.
We also know that on $N$ the eigenvalues of $\nabla \nabla f$ can only be $0$
and $\lambda .$ Thus their multiplicities are constant. Using that the rank
of $\nabla \nabla f$ is constant we see that $N$ is a submanifold whose
tangent space is given by $\ker \left( \nabla \nabla f\right) .$ This in
turn shows that $N$ is a totally geodesic submanifold.

Note that when $\lambda >0$ the minimum set $N$ is in fact compact as it
must be an Einstein manifold with Einstein constant $\lambda .$

The normal exponential map%
\begin{equation*}
\exp :v\left( N\right) \rightarrow M
\end{equation*}
follows the integral curves for $\nabla f$ or $\nabla r$ and is therefore a
diffeomorphism.

Using the fundamental equations (see \cite{Petersen}) we see that the metric
is completely determined by the fact that it is radially flat and that $N$
is totally geodesic. From this it follows that the bundle is flat and hence
of the type $N\times _{\Gamma }\mathbb{R}^{k}.$

Alternately note that radial flatness shows that all Jacobi fields along
geodesics tangent to $\nabla f$ must be of the form 
\begin{equation*}
J=E+tF
\end{equation*}%
where $E$ and $F$ are parallel. This also yields the desired vector bundle
structure.
\end{proof}

\section{Other Results}

In this section we discuss some further applications of the formulas derived
above. First we recall some technical tools. We will use the following
notation. 
\begin{equation*}
\Delta _{X}=\Delta -D_{X}
\end{equation*}

Recall the maximum principle for elliptic PDE's.

\begin{theorem}[Maximum Principle]
If $u$ is a real valued function with $\Delta_X(u) \geq 0$ then $u$ is
constant in a neighborhood of any local maximum.
\end{theorem}

The first lemma follows from Lemma \ref{Sol1}.

\begin{lemma}
If $M$ is a complete expanding or steady Ricci soliton then 
\begin{equation*}
\Delta_X|X|^2 \geq 0.
\end{equation*}
Moreover, $\Delta_X|X|^2 = 0$ if and only if $M$ is Einstein.
\end{lemma}

\begin{proof}
This follows directly from the formula $\frac{1}{2} \Delta_X|X|^2 = |\nabla
X|^2 - \lambda |X|^2$.
\end{proof}

Applying the maximum principle then shows that $|X|$ can not achieve its
maximum without being trivial.

\begin{theorem}
If $M$ is a complete expanding or steady Ricci soliton and $|X|$ achieves
its maximum then $M$ is Einstein.
\end{theorem}

Note that this clearly implies the following result for compact steady and
expanding solitons mentioned in the introduction.

\begin{corollary}
\label{Cmpct} Compact expanding or steady Ricci solitons are Einstein.
\end{corollary}

When we have a gradient soliton we use the notation $\Delta _{X}=\Delta _{f}$%
. From Lemma \ref{Sol2} we also have the following inequality

\begin{lemma}
If $M$ is a steady gradient soliton or an expanding gradient soliton with
nonnegative scalar curvature, then 
\begin{equation*}
\Delta _{f}(\mathrm{scal})\leq 0.
\end{equation*}%
Moreover, $\Delta _{f}(\mathrm{scal})=0$ if and only if $M$ is Ricci flat.
In particular, the only expanding gradient soliton with nonnegative scalar
curvature and $\Delta _{f}(\mathrm{scal})=0$ is the Gaussian.
\end{lemma}

\begin{proof}
This follows easily from the equation 
\begin{equation*}
\frac{1}{2}\Delta _{f}\mathrm{scal}=-|\mathrm{Ric}|^{2}+\lambda \mathrm{scal.%
}
\end{equation*}%
That a Ricci flat expanding soliton must be a Gaussian is just Proposition %
\ref{Ein-Sol}.
\end{proof}

Now from the maximum principle we have that the scalar curvature cannot have
a minimum.

\begin{theorem}
\label{gradmin}

\begin{enumerate}
\item A steady gradient soliton whose scalar curvature achieves its minimum
is Ricci flat.

\item An expanding gradient soliton with nonnegative scalar curvature
achieving its minimum is a Gaussian.
\end{enumerate}
\end{theorem}

For gradient solitons there is a naturally associated measure $dm=e^{-f}d%
\mathrm{vol}_{g}$ which makes the operator $\Delta _{f}$ self-adjoint.
Namely the following identity holds for compactly supported functions. 
\begin{equation*}
\int_{M}\Delta _{f}(\phi )\psi dm=-\int_{M}\langle \nabla \phi ,\nabla \psi
\rangle dm=\int_{M}\phi \Delta _{f}(\psi )dm.
\end{equation*}

The measure $dm$ also plays an important role in Perel'man's entropy
formulas for the Ricci flow \cite{PerelmanI}. In \cite{Yau1976}, Yau proves
that on a complete Riemannian manifold any $L^{\alpha }$, positive,
subharmonic function is constant. The argument depends solely on using
integration by parts and picking a clever test function $\phi $. Therefore,
the argument completely generalizes to the measure $dm$ and operator $\Delta
_{f}$. Specifically the following $L^{\alpha }$ Liouville theorem holds.

\begin{theorem}[Yau]
\label{IntMaxPrin} Any nonnegative real valued function $u$ with $%
\Delta_f(u)(x) \geq 0$ which satisfies the condition 
\begin{equation}  \label{IntCondition}
\lim_{r \rightarrow \infty} \left( \frac{1}{r^2}\int_{B(p,r)} u^{\alpha} dm
\right) = 0
\end{equation}
for some $\alpha > 1$ is constant.
\end{theorem}

Define $\Omega _{u,C}=\{x:u(x)\geq C\}$. If we only have a bound on the $f$%
-Laplacian on $\Omega _{u,C}$ then we can apply the $L^{\alpha }$ Liouville
theorem to prove the following corollary.

\begin{corollary}
\label{cor1} If $\Delta_f(u)(x) \geq 0$ for all $x \in \Omega_{u,C}$ and $u$
satisfies (\ref{IntCondition}) then $u$ is either constant or $u \leq C$.
\end{corollary}

\begin{proof}
Apply Theorem \ref{IntMaxPrin} to the function $(u-C)_{+}=\max \{u-C,0\}$.
Then $(u-C)_{+}$ is constant which implies either $u\leq C$ or $u$ is
constant.
\end{proof}

One can also derive upper bounds on the growth of the measure $dm$ from the
inequality $\mathrm{Ric}+\mathrm{Hess}f\geq \lambda g$ see \cite{Morgan2005,
Wei-Wylie}. In particular, when $\lambda >0$, the measure is bounded above
by a Gaussian measure. Combining this estimate with the $L^{\alpha }$
maximum principle gives the following strong Liouville theorem for shrinking
gradient Ricci solitons.

\begin{corollary}
\cite{Wei-Wylie} \label{Louiville} If $M$ is a complete manifold satisfying 
\begin{equation*}
\mathrm{Ric} + \mathrm{Hess} f \geq \lambda g
\end{equation*}
for $\lambda > 0$ and $u$ is a real valued function such that $\Delta_f(u)
\geq 0$ and $u(x) \leq Ke^{\beta d(p,x)^2}$ for some $\beta < \lambda$ then $%
u$ is constant.
\end{corollary}

A similar result, under the additional assumption that $\mathrm{Ric}$ is
bounded above, is proven by Naber \cite{Naber}. In fact, one can see
immediately from the equation 
\begin{equation*}
\Delta _{f}(\mathrm{scal})=\sum \lambda _{i}(\lambda -\lambda _{i})
\end{equation*}%
that if $0\leq \mathrm{Ric}\leq \lambda $ for a shrinking soliton then $%
\mathrm{scal}$ is bounded, nonnegative, and has $\Delta _{f}(\mathrm{scal}%
)\geq 0$. Therefore it is constant and we have Lemma \ref{NaberLemma}. Using
the Liouville theorem the following improvement of Proposition \ref{Prop} is
also true for shrinking gradient solitons.

\begin{theorem}
If $\mathrm{scal}$ is bounded, then 
\begin{equation*}
0\leq \inf_{M}\mathrm{scal}\leq n\lambda .
\end{equation*}%
Moreover, if $\mathrm{scal}\geq n\lambda ,$ then $M$ is Einstein.
\end{theorem}

\begin{proof}
First suppose that $\mathrm{scal}\geq n\lambda $. By the Cauchy-Schwarz
inequality 
\begin{eqnarray*}
\Delta _{f}(\mathrm{scal}) &=&-|\mathrm{Ric}|^{2}+\lambda \mathrm{scal} \\
&\leq &-\frac{\mathrm{scal}^{2}}{n}+\lambda \mathrm{scal} \\
&\leq &\mathrm{scal}\left( \lambda -\frac{\mathrm{scal}}{n}\right) 
\end{eqnarray*}%
So that $\Delta _{f}(\mathrm{scal})\leq 0.$ Let $K$ be the upper bound on $%
\mathrm{scal}$ then the function $u=K-\mathrm{scal}$ is bounded,
nonnegative, and has $\Delta _{f}(u)\geq 0$. So by Corollary \ref{Louiville} 
$\mathrm{scal}$ is constant and thus must be Einstein.

To see the other inequality consider that on $\Omega _{0}=\{x:\mathrm{scal}%
(x)\leq 0\}$, $\Delta _{f}(\mathrm{scal})\leq 0$, so applying Corollary \ref%
{cor1} to $-\mathrm{scal}$ gives the result.
\end{proof}

For steady and expanding gradient solitons we can also apply the $L^{\alpha
} $ Liouville theorem to the equation, $\Delta _{f}(|\nabla f|^{2})\geq 0$.

\begin{theorem}
\label{Gap} Let $\alpha > 2$. If $M$ is a steady or expanding soliton with 
\begin{equation}  \label{GapEqn}
\limsup_{r \rightarrow \infty} \frac{1}{r^2} \int_{B(p,r)} |\nabla
f|^{\alpha}e^{-f}dvol_g = 0.
\end{equation}
then $M$ is Einstein.
\end{theorem}

We think of Theorem \ref{Gap} as a gap theorem for the quantity $%
\int_{B(p,r)} |\nabla f|^{\alpha}e^{-f}dvol_g$ since, if $M$ is Einstein,
the quantity is zero.

For steady solitons $\mathrm{scal}+|\nabla f|^{2}$ is constant so if the
scalar curvature is bounded then so is $|\nabla f|$ and (\ref{GapEqn}) is
equivalent to the measure $dm$ growing sub-quadratically. Therefore, we have
the following corollary.

\begin{corollary}
\label{GrowthGap} If $M$ is a steady Ricci soliton with bounded scalar
curvature and 
\begin{equation*}
\lim_{r \rightarrow \infty} \frac{1}{r^2} \int_{B(p,r)}e^{-f} dvol_g = 0
\end{equation*}
Then $M$ is Ricci flat.
\end{corollary}

We note the relation of this result to the theorem proved by the second
author and Wei that if $\mathrm{Ric} + \mathrm{Hess} f \geq 0$ and $f$ is
bounded then the growth of $e^{-f} dvol_g$ is at least linear \cite%
{Wei-Wylie}. Since Ricci flat manifolds have at least linear volume growth
Corollary \ref{GrowthGap} implies that steady Ricci solitons with bounded
scalar curvature also have at least linear $dm$-volume growth. There are
Ricci flat manifolds with linear volume growth so Corollary \ref{GrowthGap}
can be viewed as a gap theorem for the growth of $dm$ on gradient steady
solitons.

\end{document}